\documentclass[a4paper,11pt,english]{smfart}

\usepackage{amsfonts}
\usepackage{amsmath,amsthm}
\usepackage{latexsym}
\usepackage{array}
\usepackage{amssymb}
\usepackage{graphics}
\usepackage{enumerate}
\usepackage[english]{babel}

\usepackage{color}
\definecolor{marin}{rgb}   {0.,   0.3,   0.7} 
\definecolor{rouge}{rgb}   {0.8,   0.,   0.} 
\definecolor{sepia}{rgb}   {0.8,   0.5,   0.} 
\usepackage[colorlinks,citecolor=marin,linkcolor=rouge,
            bookmarksopen,
            bookmarksnumbered
           ]{hyperref}

\theoremstyle{plain} 
\newtheorem{theorem}{Theorem}[section]
\newtheorem{lemma}[theorem]{Lemma}

 \theoremstyle{remark}
\newtheorem{remark}[theorem]{Remark}

\newcommand{\R}{  \mathbb{R}   }

\newcommand{\C}{  \mathbb{C}   }
\newcommand{\Z}{  \mathbb{Z}   }
\newcommand{\N}{  \mathbb{N}   }

\newcommand{\Fc}{  \mathcal{F}   }

\renewcommand{\O}{  \mathcal{O}   }

\newcommand{\T}{  \mathbb{T}   }

\newcommand{\dd}{  \text{d}   }

\newcommand{\ka}{  \kappa   }

\renewcommand{\phi}{  \varphi  }

\newcommand{\be}{\begin{equation}}
\newcommand{\ee}{\end{equation}}
\newcommand{\ben}{\begin{equation*}}
\newcommand{\een}{\end{equation*}}

\newcommand{\Norm}[2]{\|#1\|\left.\vphantom{T_{j_0}^0}\!\!\right._{#2}}

\numberwithin{equation}{section}

\def\mod#1{\left\lfloor#1\right\rceil}

\newcommand{\I}{i}
\newcommand{\E}{e}
\newcommand{\BK}{{\mathcal{B}^K}}
\newcommand{\BKO}{{\mathcal{B}^{\kappa}}}
\newcommand{\K}{K}
\newcommand{\indk}{k}

\newcommand{\dx}{\delta x}
\newcommand{\UK}{U}

\author{Erwan Faou}
\address{INRIA \& ENS Cachan Bretagne  \\
Avenue Robert Schumann F-35170 Bruz, France. } 
\email{Erwan.Faou@inria.fr}

 \author{Tiphaine J\'ez\'equel}
\address{INRIA \& ENS Cachan Bretagne  \\
Avenue Robert Schumann F-35170 Bruz, France. } 
\email{Tiphaine.Jezequel@inria.fr}

\title[Instabilities and resonant time steps]
{Resonant time steps and instabilities in the numerical integration of Schr\"odinger equations.}

\begin{document}

\begin{abstract}
We consider the linear and non linear cubic Schr\"odinger equations with periodic boundary conditions, and their approximations by splitting methods. 
We prove that for a dense set of  arbitrary small time steps, there exists numerical solutions leading to strong numerical instabilities preventing the energy conservation and regularity bounds obtained for the exact solution. We analyze rigorously these instabilities in the semi-discrete and fully discrete cases. 
\end{abstract}

\thanks{
}

\maketitle

\section{Introduction}

The goal of this paper is to show that resonant time-steps in the numerical integration of Schr\"odinger equation can lead to numerical instability in the sense that the qualitative behavior of the numerical solution is radically different from the one of the continuous solution. 
The equations we consider here are of the form 
\begin{equation}
\label{eq:sch}
i \partial_t u  = - \Delta u + f(x,|u|^2) u, \quad u(0,x) = u^0(x)
\end{equation}
where $u(t,x)$ depends on the time $t$ and the space variable $x \in \T = \R \slash (2\pi \Z)$ and $u_0$ a given function. The function $f$ is real valued, and we will only consider the two following cases: 
$$
f(x,|u|^2) = V(x) \quad\mbox{(linear)}\quad \mbox{and}\quad f(x,|u|^2) = \sigma|u|^2 \quad \mbox{(NLS)},
$$
where $V(x)$ is a real valued smooth potential and $\sigma \in \{ \pm 1\}$. In other words, we consider the linear Schr\"odinger equation or the cubic nonlinear Schr\"odinger equation (NLS). In both cases, the solution preserves the $L^2$ norm 
$$
\Norm{u(t,x)}{L^2} := \Big(\frac{1}{2\pi}\int_{0}^{2\pi}  |u(t,x)|^2 \dd x\Big)^{1/2} = \Norm{u^0}{L^2},\quad \mbox{for all}\quad t \in \R
$$
and the energy
$$
H(u) = \frac{1}{4\pi} \left( \int_0^{2\pi} |\partial_x u|^2 + f(x,|u|^2) |u|^2 \dd x\right).
$$

We consider the numerical approximation of \eqref{eq:sch} by splitting methods obtained by solving alternatively the free linear Schr\"odinger equation 
\begin{equation}
\label{eq:free}
i \partial_t u = - \Delta u, \quad \mbox{denoted by } u(t) = e^{it \Delta} u(0), 
\end{equation}
and the potential part 
\begin{equation}
\label{eq:varphi}
i \partial_t u = f(x,|u|^2)u, \quad \mbox{denoted by } u(t) = \varphi_t( u(0)), 
\end{equation}
where the flow $\varphi_t(u(0))$ is given explicitely by the formula
$$
\varphi_t(u(0) ) = e^{- i t f(x,|u(0,x)|^2)} u(0),
$$
owing to the fact that $|u(t,\cdot)|^2$ is constant along the solutions of \eqref{eq:varphi}. 

With these notations, the standard Lie-splitting method is defined by the formula
$$
u^{n+1} = e^{ i \tau \Delta} \circ \varphi_\tau (u^n), \quad n \geq 0,
$$
which gives approximation $u^n$ of the exact solution $u(n\tau)$. To implement this method, we use the classical pseudo-spectral Fourier method with $K$ modes. Note that this scheme preserves the $L^2$ norm for all times. 

Such schemes (as well as the symmetrized Strang splitting version) have been extensively studied recently, both in the semi and fully discrete case. 
Convergence results for fixed time horizon can be found in \cite{JL00,L08a} and \cite{F11}. Concerning the long time behavior of such methods, let us recall the main results: 

\begin{itemize}
\item In the linear case $f(x,|u|^2) = V(x)$, and in the semi-discrete setting, it has been shown in \cite{DF07} that if the potential is {\em small} and if the time step satisfies a non resonance condition of the form 
\begin{equation}
\label{eq:nonres}
\forall\, n \in \Z,\quad \left| \frac{1 - e^{i \tau n}}{\tau} \right| \geq \frac{\gamma}{|n|^\nu}
\end{equation}
then the numerical solution can be controlled over very long times with respect to the size of the potential, which implies the long time preservation of the energy. Still in the linear case, but when no smallness assumption on the potential is made, the construction of a modified energy is possible, see \cite{DF09}, but only for implicit-explicit schemes where the free flow $e^{i \tau \Delta}$ is approximated by the midpoint rule which implies a regularization in the high frequencies of the Laplace operator. This allows a control of the $H^1$ norm of low modes. 

\item In the fully discrete case, both in the linear and nonlinear cases, it has been shown in \cite{F11,FG11} that under a Courant-Friedrichs-Lewy (CFL) condition of the form $\tau K^2 < C$ for some constant $C$, it is possible to construct a modified energy that is almost preserved by the numerical solution $u^n$. In these cases, it is possible to prove that the discrete $H^1$ norm is bounded over long time interval $\tau^{-N}$ where $N$ depends on the CFL constant (and for small initial data in the nonlinear case). This in turn implies the preservation of the energy over long times. 
The CFL restriction is here assumed to avoid small denominators of the form \eqref{eq:nonres}: in the linear case for example, a such restriction allows to show that the denominator
$$
\frac{i \tau (k^2 - \ell^2) }{1 - e^{i \tau (k^2 - \ell^2)}}
$$
is uniformly bounded for $k$ and $\ell$ in $\Z$. 
Using such analysis, it can be proved in \cite{BFG} that splitting methods applied to NLS set on the real line with a discretization by finite differences and Dirichlet boundary conditions preserve the orbital stability of ground state over long times, provided that the CFL condition is satisfied. 

Finally, using different techniques (normal form, Modulated Fourier Expansion), it has been shown that in the nonlinear case, the plane wave solutions $u(t,x) = \rho e^{-i t |\rho|^2}$ of NLS are stable by small perturbation and numerical approximation in high-order Sobolev norms $H^s$ (the same holds true for any plane wave solution and the sign of $\sigma$ matters in the statements,  see \cite{FGL1,FGL2} for precise formulations). 

\end{itemize}

Related to this analysis, let us also mention the case where the nonlinear Schr\"odinger equation is perturbed by a potential: in this case, and under a non resonance condition on the frequencies of the linear operator, it can be shown that the exact solution is stable in $H^s$ norms for high Sobolev indices $s$ (see for example \cite{BG06}). In this case, the numerical reproduction has been studied in \cite{GL08a,GL08b} using Modulated Fourier Expansion and \cite{FGP1,FGP2} using normal form analysis. In each case, non resonance conditions have to be imposed on the time steps as in the previous analysis. 

As numerous numerical experiments indicate, when the time step $\tau$ violates a non resonance condition of the form \eqref{eq:nonres}, a drift in the energy can be generally observed (see for instance \cite{DF07,F11}). However all the previous results concern precisely cases where such phenomenon do not appear, and deal with stability results. 

The goal of this paper is precisely to show that when $\tau$ is {\em resonant}, that is a rational multiple of $2 \pi$ of the form $\tau = 2 \pi \frac{p}{q}$ with $(p,q) \in \Z$, then we can construct situations (potential, initial value) such that the energy is not preserved by the numerical scheme.

The main idea is that for such time steps, the space of the $2\pi/q$-periodic functions 
\begin{equation}
\label{eq:Wq}
W_q:= \{u\in H^1(\T) \ |  \ \hat u(k)=0 \text{ for all } k\neq 0 \mod q \}
\end{equation}
is invariant by the linear flow $e^{i \tau \Delta}$, reducing the discrete dynamics to the potential part. 

In the semi-discrete case, we prove that the energy actually tends to infinity. In the fully discrete case, we prove that the energy can grow up to the maximum possible $\delta x^{-2}$ where $\delta x$ is the mesh stepsize, which cancels any hope of energy conservation over long times. 

In the semi-linear case, it is interesting to mention that the examples given in this paper show that the CFL given in \cite{F11} to ensure the preservation of $H^1$ norm is sharp (see Remark \ref{rk38}). 

\section{Semi-discrete case}

\subsection{Notations}

For a given funtion $f$ defined on the torus, we denote by 
$$
\hat u (k) =  \frac{1}{2\pi} \int_0^{2\pi} u(x) e^{- i k x} \dd x
$$
its Fourier transform, for $k \in \Z$. We then define the Sobolev norms 
$$
\Norm{u}{H^s} := \left( \sum_{k \in \Z} ( 1 + |k|^2)^{s} |\hat u(k)|^2  \right)^{1/2}. 
$$

\subsection{Flow of the potential part}
The following result shows that the solution of the potential part \eqref{eq:varphi} generically does not preserve the $H^1$ norm. 
\begin{lemma}
\label{eq:lem1}
Let $u(x)$ and  $V(x)$ in $H^1$.  Then we have
$$
\Norm{ e^{i t V( x)} u( x) }{H^1} \geq |t| \Norm{V'(x) u(x)}{L^2} - \Norm{u}{H^1}
$$
In particular, if 
$$
\Norm{V'(x) u(x)}{L^2} \neq 0. 
$$
Then we have 
\begin{equation}
\label{eq:growth}
\lim_{t \to \infty} \Norm{e^{it V(x)} u(x)}{H^1} = +\infty. 
\end{equation}
\end{lemma}

\begin{proof}
The result comes directly from the calculus
$$
\partial_x ( e^{i t V( x)} u( x) )   =  (it) V'(x) e^{i t V( x)} u(x) + e^{i t V(x)} u'(x). 
$$
\end{proof}

\begin{remark}
Using the Faa-di Bruno formula, we easily get that more generally, $\Norm{e^{i t V( x)} u( x) }{H^s}$ behaves like $t^s$ when $t \to \infty$. 
\end{remark}
\subsection{Linear equation}
We consider the linear Schr\"odinger equation 
\begin{equation}
\label{eq:LS}
i \partial_t u(t,x)  = -\Delta u(t,x) + V(x) u(t,x), \quad u(0,x) = u^0(x)
\end{equation}
where $V(x)$ is a given potential. 
The energy in this case is given by 
\begin{equation}
\label{eq:HV}
H(u) = \frac{1}{4\pi} \Big( \int_0^{2\pi} | \partial_x u(x) |^2 + V(x) |u(x)|^2 \dd x \Big), 
\end{equation}
which is preserved along the exact flow of \eqref{eq:LS}. 

\begin{theorem}
Let $(p,q) \in \Z \times \N^*$ and $V \in W_{q}$. 
For all $n$, and $u^0(x) \in H^1$, we define the sequence $u^n$ by the induction formula
$$
u^{n+1} = \exp( i \tau \Delta ) \circ \exp( -  i \tau V(x)) u^{n}, \quad n \geq 0, \quad \tau = 2\pi \frac{p}{q}. 
$$
If $\Norm{V'(x) u^0(x) }{L^2}=c_0 \neq 0 $, then there exist some constants $c,c'$ such that 
$$
\Norm{u^n}{H^1} \geq c_0n\tau-c , \quad \mbox{and} \quad H(u^n) \geq \frac{1}{2}(c_0n\tau-c')^2, 
$$
where $H(u)$ is the energy defined in \eqref{eq:HV}. In particular,
$$
\lim_{n \to \infty} \Norm{u^n}{H^1} = + \infty, \quad \mbox{and} \quad \lim_{n \to \infty} H(u^n) = +\infty. 
$$
\end{theorem}
\begin{proof}
Given that $V\in W_q$, we have $\hat V(k) = 0$ for $k \neq 0 \mod q$. 
Let us consider the commutator $[e^{2i\pi \frac{p}{q}\Delta}, V] = e^{2i\pi \frac{p}{q}\Delta}V - V e^{2i\pi \frac{p}{q}\Delta}$. Its coefficients in terms of Fourier operator are given by 
\begin{eqnarray*}
([e^{2i\pi \frac{p}{q}\Delta}, V])_{k\ell} &=& \hat V(k -\ell) (e^{2i\pi \frac{p}{q} k^2} - e^{2i\pi \frac{p}{q}\ell^2})  \\
&=& e^{2i\pi \frac{p}{q}\ell^2} \hat V(k -\ell) (e^{2i\pi \frac{p}{q}(k^2 - \ell^2)} - 1). 
\end{eqnarray*}
When $k - \ell \neq 0 \mod q$, the term is zero because $\hat V(k- \ell) = 0$. When $k - \ell = 0 \mod q$, then $(k^2 - \ell^2) = m q (k + \ell)$ for some number $m \in \Z$,  and 
$$
2\pi \frac{p}{q} (k^2 - \ell^2) = 2\pi p m (k + \ell) \in 2 \pi \Z, 
$$
so that the commutator vanishes again. 
Hence we have $[e^{2i\pi \frac{p}{q}\Delta}, V] = 0$ from which we easily deduce that $[e^{2i\pi \frac{p}{q} \Delta}, e^{-2i\pi \frac{p}{q} V}] = 0$. 
This implies that 
$$
u^{n} = e^{2in\pi \frac{p}{q}\Delta} \circ e^{- 2in\pi \frac{p}{q}V} u^0. 
$$
As $e^{i \tau \Delta}$ is an isometry for the $H^1$ norm, this shows that 
$$
\Norm{u^{n}(x)}{H^1} = \Norm{e^{- 2in\pi \frac{p}{q}V } u}{H^1}.
$$
The first result is then given by Lemma \ref{eq:lem1}. The second part can be easily proved using the fact that 
$$
H(u) \geq \frac{1}{2} \Norm{u}{H^1}^2 - C \Norm{u}{L^2}^2
$$
where the constant $C$ depends on $\Norm{V}{H^1}$. The fact that $\Norm{u^n}{L^2} = \Norm{u^0}{L^2}$ then proves the result. 
\end{proof}

\subsection{Nonlinear equation}

We consider now the equation 
$$
i \partial_t u(t,x)  = -\Delta u(t,x) + \sigma|u(t,x)|^2 u(t,x), 
$$
where $\sigma \in \{\pm 1\}$. 
The energy associated with this equation is now given by 
\begin{equation}
\label{eq:HNLS}
H(u) = \frac{1}{4\pi} \Big( \int_0^{2\pi} | \partial_x u(x) |^2 + \frac\sigma{2} |u(x)|^4 \dd x \Big). 
\end{equation}
\begin{lemma}
Let $(p,q) \in \Z \times \N^*$. Then for all $u \in W_{q}$, we have 
$$
\exp\Big( 2i \pi \frac{p}{q} \Delta\Big) u = u \quad \mbox{and} \quad \exp\Big( 2i \pi \frac{p}{q^2} \Delta\Big) u = u
$$
\end{lemma}
\begin{proof}
By assumption, we can expand $u$ in Fourier series as 
$$
u(x) = \sum_{n \in \Z} \hat u(qn) e^{i q n x}. 
$$ 
Let $v := \exp( 2i \pi \frac{p}{q} \Delta) u$. The fonction $v$ can be written explicitely
$$
v(x) = \sum_{n \in \Z} \hat u(qn) e^{i q n x - 2i \pi \frac{p}{q} |qn|^2} = \sum_{n \in \Z} \hat u(qn) e^{i q n x - 2i \pi |p|q n^2} = u(x). 
$$
This proves the first part of the Lemma. The second is proved similarly. 
\end{proof}

Recall that for a function $u^0$, we denote by $\varphi_t(u^0)$ the solution of 
$$
i \partial_t u(t,x) = \sigma |u(t,x)|^2 u(t,x), \quad u(0,x) = u^0(x),
$$
and we have explicitely, 
\begin{equation}
\label{eq:flown}
u(t,x) = \exp(-i \sigma t |u^0(x)|^2) u^0(x). 
\end{equation}
We prove the following: 
\begin{theorem}
Let $(p,q) \in \N^*\times \Z$. For a given smooth function $u^0(x)$, we define for all $n$ the functions $u^n$ defined by 
$$
u^{n+1} = \exp( i \tau \Delta ) \circ \varphi_{\tau }(u^{n}), \quad n \geq 0, \quad \tau = 2\pi \frac{p}{q} \text{ or } 2\pi \frac{p}{q^2}. 
$$
If $u^0 \in W_{q}$ satisfies, 
$$
\Norm{u^0 \partial_x (|u^0|^2)}{L^2} =c_0\neq 0, 
$$
then there exist some constants $c,c'$ such that 
$$
\Norm{u^n}{H^1} \geq c_0n\tau-c , \quad \mbox{and} \quad H(u^n) \geq \frac{1}{2}(c_0n\tau-c')^2, 
$$
where $H(u)$ is the energy defined in \eqref{eq:HV}. In particular, 
$$
\lim_{n \to \infty} \Norm{u^n}{H^1} = +\infty \quad \mbox{and} \quad \lim_{n \to \infty} H(u^n) = +\infty. 
$$
where $H(u)$ is the energy \eqref{eq:HNLS}. 
\end{theorem}
\begin{proof}
From the expression \eqref{eq:flown} and the fact that $H^1$ is an algebra in dimension one\footnote{That is there exists $C$ such that  $\Norm{uv}{H^1} \leq C \Norm{u}{H^1} \Norm{v}{H^1}$ for all functions $u$, $v$. },  we observe that $\varphi_t$ maps $W_{q}$ to itself. Hence by the previous Lemma, we have that 
$$
u^{n} = \varphi_{\tau}(u^{n-1}) = \varphi_{n\tau}(u^0) = \exp(-i \tau n |u^0(x)|^2) u^0(x).
$$
The growth of the $H^1$ norm is then easily obtained using Lemma \ref{eq:lem1}. 

To conlude, we use the Gagliardo-Nirenberg inequality stating that 
$$
\Norm{u}{L^4}^4 \leq C \Norm{u}{H^1} \Norm{u}{L^2}^3, 
$$
for some constant $C$, where 
$$
\Norm{u}{L^4} = \left(\frac{1}{2\pi} \int_0^{2\pi} |u(x)|^4 \dd x\right)^{1/4}. 
$$
Note that this inequality is a direct consequence of the standard inequality 
\begin{equation}
\label{eq:ulinf}
\Norm{u}{L^\infty}^2 \leq C \Norm{u}{L^2} \Norm{u}{H^1}
\end{equation}
whose proof is left to the reader. 
This implies that 
$$
H(u) \geq \frac12 \Norm{u}{H^1}^2 - C \Norm{u}{H^1}\Norm{u}{L^2}^3
$$
which gives the result, as the $L^2$ norm of the semi-discrete solution is preserved. 
\end{proof}

\section{Fully discrete case}

\subsection{Notations, preliminary results}

We now consider the case of fully discrete splitting, where the space discretization is done using a Fourier  pseudo-spectral method. 
\subsubsection*{Space discretization : grid points}
Let us first consider the space discretization of \eqref{eq:sch}. For an even integer $K$, we set 
$$
\BK = \Big\{ - \frac{K}{2}, \ldots, \frac{K}{2} - 1\Big\}. 
$$
We define the grid points $x_j = \frac{2\pi j}{K}$, for $j \in \BK$. These grid points belong to the interval $[-\pi,\pi[$. 
Moreover, we set $\dx:=\frac{2\pi}{\K}$. 

\subsubsection*{Fourier transform} 
With this grid, we associate the discrete Fourier transform $\Fc_K: \C^{\BK} \to \C^{\BK}$ defined by the formula
$$
(\Fc_K v)_j = \frac1K \sum_{k\in \BK} e^{-2i\pi jk/K} v_k. 
$$
Its inverse is given by 
$$
(\Fc_K^{-1} v)_k = \sum_{j \in \BK} e^{2i\pi kj/K} v_j. 
$$
Note that we usually use the same notation to denote the transformation $\Fc_K$ and its matrix representation as linear transformation of $\C^{\BK}$. 
With this convention, it is easy to see that $\sqrt{K} \Fc_K$ is a unitary transformation of $\C^K$ equipped with the discrete $L^2$ norm 
\begin{equation}
\label{eq:ell2}
\Norm{U}{\ell^2} : = \frac{2\pi}{K} \sum_{k \in \BK} |U_k|^2. 
\end{equation}

\subsubsection*{The pseudo-spectral Fourier method}
We search for a function 
$$
U_K(t,x) = \sum_{j \in \BK} e^{i j x} \hat U_j(t)
$$
satisfying \eqref{eq:sch} at each grid points $x_k$, $k \in \BK$, that is the relation 
$$
\forall\, k \in \BK, \quad \partial_t U_K(t,x_k) = - \Delta U_K(t,x_k) + f(x_k,|U_K|^2) U_K(t,x_k), 
$$
where $f(x_k,|U_K|^2) = V(x_k)$ in the linear case, and $f(x_k,|U_K|^2) = \sigma |U_K(t,x_k)|^2$ in the nonlinear case. 
 The coefficients $\hat U = (\hat U_j)_{j \in \BK}$ are the discrete Fourier coefficients, and 
setting the notation $U = (U_k)_{k \in \BK}$ with $U_k(t) = U_K(t,x_k)$, we have the relation $\hat U(t) = \Fc_K U(t)$. Moreover the vector $U(t)$ satisfies the following system of ordinary differential equations
$$
i \frac{\dd}{\dd t} U(t) = \Delta^K U(t) + f^K(U) U(t)
$$
where 
$$
\Delta^K = \Fc_K^{-1} D^K \Fc_K\quad \mbox{where} \quad D^K = \mathrm{diag} (|j|^2), \quad j \in \BK, 
$$
and 
$$
f^K(U) = \mathrm{diag}( f( x_k,|U_K|^2(t,x_k)), \quad k \in \BK. 
$$
\subsubsection*{Energy} 
The energy associated with the previous system is given by 
\begin{equation}
\label{eq:HK}
H^K = \frac{\pi}{K} \Big( U^T \Delta^K U + U^T f^K(U) U\Big).
\end{equation}
Alternatively, the equation satisfied by $\hat U$ is given by 
$$
\frac{\dd}{\dd t} \hat U(t) = D^K \hat U(t) + ( \Fc_K f^K(U) \Fc_K^{-1}) \hat U(t). 
$$

\subsubsection*{The fully discrete splitting method} 
Applied to the space discretized equation, the splitting method consists in solving alternatively 
\begin{equation}\label{eq:DeltaK}
i \frac{\dd}{\dd t} \hat U(t) = D^K \hat U(t), \quad \Longleftrightarrow\quad \forall\, j \in \BK, \quad \hat U_j(t)  = \exp( - i \tau |j|^2) \hat U_j(0)
\end{equation}
which we denote by $U(t) = e^{it\Delta^K} U(0)$, and 
\begin{multline}
\label{eq:vd}
i \frac{\dd}{\dd t} U(t) = f^K(t)  U(t), \\\quad \Longleftrightarrow\quad \forall\, k \in \BK, \quad  U_k(t)  = \exp( - i t f(x_k,|U_k(0)|^2)) U_k(0). 
\end{multline}
and to use the discrete Fourier transforms $\Fc_K$ and $\Fc_K^{-1}$ to switch from the $U_k$ to the $\hat U_j$.  The approximation $U^n$ of $U(n\tau)\in \C^{\BK}$ is defined by the induction formula
$$
U^{n+1}:=e^{i\tau\Delta^K}\circ  \varphi_\tau (U^n),
$$
where $\varphi_t(U)$ is the solution of equation \eqref{eq:vd} at time $t$ with the initial condition $U$. Note that this scheme preserves the discrete $\ell^2$ norm \eqref{eq:ell2}. 

\subsubsection*{Resonant space}
We recall that in the semi-discrete case, we show some resonance phenomena when the potential $V$ (in the linear case) or the initial condition (in the nonlinear case) belong to the space of functions $W_q$ defined in \eqref{eq:Wq}. In the fully discrete case, in order to obtain discrete spaces $W_q$, we assume that $K$ is a multiple of $q$, namely of the form $K=\ka q$ with $\ka \in2\Z$. We then introduce the spaces of functions $W_{\ka,q}$ defined as follow
\begin{equation}\label{eq:Wkq}
W_{\ka,q}:=\{U_j \in \C^{\BK} \ | \  (\mathcal{F}_K U)_j =0 \text{ for } j\neq 0 \mod q\}.
\end{equation}
Observe that when $K=\ka q$, the elements $U$ of $W_{\ka,q}$ are of the form
$$
U_\indk=\sum_{j \in \BKO}e^{ijqx_\indk}\hat U_{jq}.
$$

\subsubsection*{Discrete $h^1$ norms, estimates}
Following the same approach as in the previous section, it is clear that the use of resonant time step will cancel the averaging effect of the free flow, and we will be lead to analyze the growth of energy along solutions of \eqref{eq:vd}. As the reader can see, this equation is easily written in terms of $U$ and not in terms of $\hat U$, while for the kinetic energy, it is the opposite. That is why we introduce the following norms: 
\begin{equation}
\label{eq:n1}
\Norm{U}{h^1}^2 = \frac{2\pi}{K}\sum_{k \in \BK}  \left |Ê\frac{U_{k+1} - U_k}{\delta x}\right|^2 \quad \mbox{and} \quadÊ T^K(\hat U)  = \frac{1}{K} \sum_{k \in \BK} |k|^2|\hat U_k|^2, 
\end{equation}
with the convention that $U_{K/2} = U_{-K/2}$ (periodicity). 
Note that the last term represents the kinetic energy associated with the energy $H^K$ given in \eqref{eq:HK}, and that 
$$
\Norm{\Fc_K^{-1} \hat U}{h^1}^2 \neq T^K(\hat U). 
$$
However, we have the following 
\begin{lemma}\label{Lem:TK}
There exist constants $c$ and $C$ such that for all $U \in \C^{\BK}$, we have 
$$
c \Norm{U}{h^1}^2 \leq T^K(\Fc_K U) \leq C \Norm{U}{h^1}^2. 
$$
\end{lemma}
\begin{proof}
We write explicitly (using the periodicity condition)
$$
 U_{k+1} - U_k = \sum_{j \in \BK} ( e^{i (\delta x ) (k +1) j } - e^{i (\delta x) kj}) \hat U_j = \sum_{j \in \BK} ( e^{i (\delta x) j }  -1) e^{i (\delta x) k j } \hat U_j, 
 $$
 which yields, for $k \in \BK$, 
 $$
| U_{k+1} - U_k|^2 =  \sum_{j,j' \in \BK} ( e^{i (\delta x) j }  -1) ( e^{- i (\delta x) j' }  -1)  e^{i (\delta x) k (j - j')} \hat U_j \overline{\hat U_{j'}}. 
 $$
Summing in $k$, and using the fact that 
 $$
 \sum_{k \in \BK} e^{i (\delta x) k m } = \left\{\begin{array}{rl} 0&  \mbox{if}\quad  m \neq 0 \mod{K}\\[2ex]
 1 & \mbox{if} \quad m = 0 \mod K, 
 \end{array} 
 \right.
 $$
 we obtain
  $$
\sum_{k \in \BK}| U_{k+1} - U_k|^2 =  \sum_{j \in \BK} | e^{i (\delta x) j }  -1 |^2 |\hat U_j|^2
 $$
 From this relation, we obtain 
 $$
 \Norm{U}{h^1}^2 =  \delta x \sum_{j \in \BK} \left| \frac{e^{i (\delta x) j }  -1 }{(\delta x)  j}\right|^2 |j|^2 |\hat U_j|^2
 $$
 and we easily conclude by using the fact that the function $x \mapsto |\frac{e^{i x} - 1}{x}|$ is uniformly bounded from above and below on the interval $[-\pi,\pi[$.  
 \end{proof}

\subsubsection*{Discrete Gagliardo-Nirenberg}
Finally, we will also need the following discrete version of the Gagliardo-Nirenberg inequality: 

\begin{lemma}\label{Lem:GNdiscret}
There exists a constant such that 
for all $K$ and  all $U = (U_k)_{k \in \BK} \in \C^{\BK}$, we have 
$$
\delta x \sum_{k \in \BK} |U_k|^4 \leq C \Norm{U}{h^1} \Norm{U}{\ell^2}^3. 
$$
\end{lemma}
\begin{proof}
The proof of this inequality is the same as in the continuous case, or can be deduced from \eqref{eq:ulinf} using piecewise linear interpolation for example. 
\end{proof}

\subsection{Flow of the fully discrete potential part}
In the fully discrete case, it is easy to see that there exists a constant $C$ such that for all $U \in \C^{\BK}$, 
\begin{equation}
\label{eq:inverse}
\Norm{U}{h^1} \leq C K \Norm{U}{\ell^2}. 
\end{equation}
Hence, as the fully discrete splitting method described above preserves the $\ell^2$ norm of the discrete solution, we cannot observe a drift of the $h^1$ norm even for the flow of the potential part. However, we will prove below that for this discrete flow, the $h^1$ norm can become of order $K$ in a time of order $\mathcal{O}(K)$, which corresponds to a discrete version of the drift described in the previous section. 
In the following for a given vector $U = (U_k) \in \C^{\BK}$, we set 
$$
\Norm{U}{\ell^\infty} = \max_{k \in \BK} | U_k|. 
$$

\begin{lemma}\label{KeyLemmaDiscret}
Let $U = (U_k)_{k \in \BK} \in \C^{\BK}$ and $V = (V_k)_{k \in \BK} \in \R^{\BK}$. Then for $n \tau$ and $\delta x$ satisfying 
\begin{equation}
\label{eq:estn}
n \tau \delta x \Big\| \frac{V_{k+1} - V_k}{\delta x} \Big\|_{\ell^\infty} \leq \pi, 
\end{equation}
we have 
\begin{equation}
\label{eq:esti}
\| e^{i n \tau V_k} U_k \|_{h^1} \geq \frac{2}{\pi}n \tau \Big\|   \Big( \frac{V_{k+1} - V_k}{\delta x}\Big) U_k \Big\|_{\ell^2} -  \Norm{U}{h^1}
\end{equation}
\end{lemma}

\begin{proof}
We have for all $k \in \BK$, 
$$
 e^{i n \tau V_{k+1}} U_{k+1} -   e^{i n \tau V_k} U_k  = e^{i n \tau V_k}\Big( U_{k+1} - U_{k} \Big) + U_k \Big( e^{i n \tau V_{k+1}} - e^{i n \tau V_{k}}\Big). 
$$
The first term in the right-hand side gives the term $- \Norm{U}{h^1}$ in \eqref{eq:esti}. Hence we are led to prove that for all $n$, $\tau$ and $\delta x$ satisfying \eqref{eq:estn}, we have 
$$
\forall\, k \in \BK, \quad 
|e^{i n \tau (V_{k+1} - V_k)} - 1| \geq \frac{2}{\pi}n \tau |V_{k+1} - V_k|
$$
which is a consequence of the fact that 
$$
\forall\, x \in [-\pi,\pi], \quad | e^{-ix} - 1| \geq \frac{2}{\pi} |x|. 
$$
\end{proof}
As a corollary, we see that if $V_k$ and $U_k$ are given and satisfy uniform estimates of the form 
$$
\forall \delta x,\quad  \Big\| \frac{V_{k+1} - V_k}{\delta x} \Big\|_{\ell^\infty} + \Norm{U}{h^1}\leq C \quad \mbox{and} \quad \Big\|   \Big( \frac{V_{k+1} - V_k}{\delta x}\Big) U_k \Big\|_{\ell^2}\geq c
$$
we obtain 
$$
\| e^{i n \tau V_k} U_k \|_{h^1} \simeq  (\delta x)^{-1}\quad \mbox{for a time} \quad t  \simeq(\delta x)^{-1}, 
$$
which is a discrete version of Lemma \ref{eq:lem1}. Note that such estimate shows that the inequality \eqref{eq:inverse} is saturated owing to $K \simeq (\delta x)^{-1}$. 

\subsection{Linear equation}
With the notations introduced above, we consider here the particular case 
$$
f(x_k, |U_k|^2)=V(x_k) =: V_k, \quad k \in \BK, 
$$
where $V(x)$ is a given potential. In particular, $\varphi_\tau$ reads
$$
\varphi_t(U)=e^{-it \mathrm{diag} (V_k)}U.
$$
In this case it is proved in \cite{F11} that under a CFL condition, there exists a modified energy ensuring global $h^1$ bound for the numerical solution (see Theorem V.13 and Corollary V.14 in \cite{F11}). 

\begin{theorem}
Let $(p,q) \in \Z \times \N^*$, $\kappa \in 2\Z$, $K = \kappa q$  and $V \in W_{\ka,q}$. 
For all $n$, and all $\UK^0 \in \C^{\BK} $, we define the sequence $\UK^n$ by the induction formula
$$
\UK^{n+1} = \exp( i \tau \Delta^K ) \circ \exp( -  i \tau \mathrm{diag} (V_k)) \UK^{n}, \quad n \geq 0, \quad \tau = 2\pi \frac{p}{q}. 
$$
We assume that there exist $C_0, C_1, C_2$ and $c_0$ such that for all $K$, 
\begin{equation}\label{eq:hyp1}
\Norm{V}{\ell^\infty}\leq C_0, \quad \Big\| \frac{V_{k+1} - V_k}{\delta x} \Big\|_{\ell^\infty} \leq C_1, \quad  \Norm{U^0}{h^1}\leq C_2,
\end{equation}
and 
\begin{equation}\label{eq:hyp2}
c_0\leq \Big\|   \Big( \frac{V_{k+1} - V_k}{\delta x}\Big) U^0_k \Big\|_{\ell^2}.
\end{equation}
Then there exist some constants $c,c',c'', c'''$ such the for for all $n, \tau, \dx$ satisfying the relation
$$
n\tau\leq \frac{\pi}{\delta x \, C_1},
$$
the sequence $\UK^n$ verifies the estimates
$$
\Norm{U^n}{h^1}\geq cn\tau -c', \quad\quad H^K(U^n)\geq (c''n\tau-c''')^2.
$$
\end{theorem}
\begin{remark}
Suppose that $U^0_k = u^0(x_k)$ and $V_k = V(x_k)$ are the space discretizations of some $L^2$ functions $u^0$ and $V$. Then on one hand, observe that \eqref{eq:hyp1} holds for all $K$ as soon as $\Norm{V}{L^\infty}, \Norm{V'}{L^\infty}$ and $\Norm{u^0}{L^2}$ are finite. On the other hand, to get \eqref{eq:hyp2} for all $K$, it is sufficient to assume that $\Norm{u^0V'}{L^2}\neq 0$.
\end{remark}

\begin{proof}
Given that $V\in W_{\ka,q}$, we have $\hat V_k:=\left(\mathcal{F}_K(V)\right)_\indk = 0$ for $\indk \neq 0 \mod q$. Let us consider the linear operators $U\mapsto VU$ and $U\mapsto e^{2i\pi \frac{p}{q}\Delta^K}U$ expressed in Fourier basis. Then the commutator reads 
\begin{eqnarray*}
([e^{2i\pi \frac{p}{q}\Delta}, V])_{jk} &=& (\hat{V}_{j-k}+\hat{V}_{j-k+K}
 +\hat{V}_{j-k-K}) (e^{2i\pi \frac{p}{q} j^2} - e^{2i\pi \frac{p}{q}k^2})  \\
&=&e^{2i\pi \frac{p}{q}k^2}(\hat{V}_{j-k}+\hat{V}_{j-k+K}+\hat{V}_{j-k-K}) (e^{2i\pi \frac{p}{q}(j^2-k^2)}-1) . 
\end{eqnarray*}
On one hand, when $j-k \neq 0 \mod q$, from $V\in W_{\ka,q}$ and $K=\ka q\equiv 0 \mod q$, the term is zero. On the other hand, when $j - k = 0 \mod q$, then $(j^2 - k^2) = \alpha q (j + k)$ and 
$$
2\pi \frac{p}{q} (j^2 - k^2) = 2\pi \alpha p (j + k) \in 2 \pi \Z. 
$$
so that the commutator vanishes again. 
Hence we have $[e^{2i\pi \frac{p}{q}\Delta^K}, V] = 0$ from which we easily deduce that $[e^{2i\pi \frac{p}{q} \Delta^K}, e^{-2i\pi \frac{p}{q} \mathrm{diag}(V_k)}] = 0$. 
This implies that 
$$
\UK^n = e^{2in\pi \frac{p}{q}\Delta^K} \circ e^{- 2in\pi \frac{p}{q} \mathrm{diag}(V_k)} \UK^0. 
$$
In the continuous case above, at this step of the proof we used that $e^{2in\pi \frac{p}{q}\Delta}$ is an isometry for the $H^1$-norm. Here, we use $T^K$, which is not a norm but is preserved by the flow of $\Delta^K$, and is linked to $\Norm{\cdot}{h^1}$ by the result of Lemma \ref{Lem:TK}.  We get
\begin{equation}
\Norm{\UK^n}{h^1}^2\geq \frac{1}{C}T^K\circ \Fc_K(e^{2in\pi\frac{p}{q}V}\UK^0)\geq \frac{c}{C}\Norm{e^{2in\pi\frac{p}{q}V}\UK^0}{h^1}^2, 
 \end{equation}
where $c$ and $C$ are the constants introduced in Lemma \ref{Lem:TK}. Finally, from Lemma \ref{KeyLemmaDiscret} together with the assumptions \eqref{eq:hyp1} and \eqref{eq:hyp2}, we obtain

$$\Norm{\UK^n}{h^1}\geq \frac{2\sqrt{c}}{\pi\sqrt{C}}n\tau \Big\|   \Big( \frac{V_{k+1} - V_k}{\delta x}\Big) U^0_k \Big\|_{\ell^2}-\Norm{\UK^0}{h^1}\geq cn\tau-c'.$$

This proves the first estimate of the lemma.  To get the estimate of $H^K$, observe that
$$
H^K(U)= \pi T^K(\mathcal{F}_K U) +\frac{\pi}{K}U^T \mathrm{diag}(V_k) U \geq c\Norm{U}{h^1}^2 - \Norm{V}{\ell^\infty}\Norm{U}{\ell^2}^2,
$$
where the inequality comes from Lemma \ref{Lem:TK}. Thus, from assumption \eqref{eq:hyp1} and given that $\Norm{U^n}{\ell^2}$ is constant, we get
$$
H^K(U^n)\geq (c''n\tau - \sqrt{c}c')^2 -  C_0\Norm{U^0}{\ell^2}^2\geq (c''n\tau -c''')^2.
$$
\end{proof}

\subsection{Nonlinear equation}
We consider the cubic nonlinear Schr\"odinger equation on the grid $(x_\indk)_{\indk\in\BK}$. This means that with the notations introduced above, we consider here the particular case 
$$
f(x_k, |U_k|^2)=\sigma|U_k|^2.
$$
where $\sigma \in \{\pm 1\}$. Here $\varphi_t(\UK_0)$ has the following explicite form
\begin{equation}\label{Phit}
\varphi_t(\UK^0)(x_\indk) = \exp(-\sigma i t |\UK^0(x_\indk)|^2) \UK^0(x_\indk), \quad \forall \indk\in\BK. 
\end{equation}
In this case, it is shown in \cite{F11} that under a CFL condition of the form 
\begin{equation}
\label{cflf11}
\tau K^2 < \frac{8\pi}{N+1}
\end{equation}
it is possible to construct a modified energy, and obtain global $h^1$ bound for small solutions for a time of order $\tau^{-N}$ (see Theorem VI.8 and Corollary VI.10 in \cite{F11}). 
Here, we prove the following: 
\begin{theorem}
Let $(p,q) \in \N^*\times \Z$ and $K=\ka q$. For a given $\UK^0$ in $\C^{\BK}$, we define for all $n$ the $\UK^n$ by 
$$ \UK^{n+1} = \exp( i \tau \Delta^K ) \circ \varphi_{\tau }(\UK^n), \quad n \geq 0, \quad \tau = 2\pi \frac{p}{q}.$$
We suppose that $\UK^0 \in W_{\ka,q}$ and that there exist $C_0, C_1, C_2$ and $c_0$ such that for all $K$, 
\begin{equation}\label{eq:hyp1bis}
\Norm{\UK^0}{\ell^\infty}\leq C_0, \quad \Big\| \frac{|\UK^0_{k+1}|^2 - |\UK^0_k|^2}{\delta x} \Big\|_{\ell^\infty} \leq C_1, \quad  \Norm{U^0}{h^1}\leq C_2,
\end{equation}
and 
\begin{equation}\label{eq:hyp2bis}
c_0\leq \Big\|   \Big( \frac{|\UK^0_{k+1}|^2 - |\UK^0_k|^2}{\delta x}\Big) U^0_k \Big\|_{\ell^2}.
\end{equation}
Then there exist some constants $c,c',c'',c'''$ such that for all $n, \tau, \dx$ satisfying the relation
$$
n\tau\leq \frac{\pi}{\dx \ C_1},
$$
the sequence $\UK^n$ verifies the estimates
$$
\Norm{U^n}{h^1}\geq cn\tau -c', \quad\quad H^K(U^n)\geq (c''n\tau-c''')^2.
$$
Moreover, the same result holds true for time steps of the form 
$$
\tau = 2\pi \frac{p}{q^2}. 
$$
\end{theorem}
\begin{remark}
Suppose that $U^0_k = u^0(x_k)$ is the space discretization of some $L^2$ function $u^0$. Then on one hand, \eqref{eq:hyp1bis} holds for all $K$ as soon as $\Norm{u^0}{L^{\infty}}, \Norm{(u^0)'}{L^{\infty}}$ and $\Norm{u^0}{L^2}$ are finite. On the other hand, to get \eqref{eq:hyp2bis} for all $K$, it is sufficient to assume that $\Norm{u^0(u^0)'}{L^2}\neq 0$.
\end{remark}
\begin{remark}
\label{rk38}
Taking a stepsize of the form $\tau =  \frac{2\pi}{q^2}$ and $K = 2 q$ we obtain a CFL number $\tau K^2 = 8\pi $ and we obtain a linear drift in $h^1$ norm for times up to $(\delta x)^{-1} \simeq K \simeq \tau^{-1/2}$. 
This can be compared with the CFL condition \eqref{cflf11} required to construct a modified energy and obtain $H^1$ bound for a time of order $\tau^{-N}$. In other words, this shows that the CFL condition  \eqref{cflf11} is sharp.  
\end{remark}
\begin{proof}
We observe first that  for all $t$, $\varphi_t$ maps $W_{\ka,q}$ to itself. 

Secondly, let us prove that for all $\UK$ in $W_{\ka,q}$, the flow of $\Delta^K$ satisfies
$$\exp(2\I\pi\frac{p}{q}\Delta^K)\UK=\UK.$$ 
To show that, we recall that any $\UK$ of $W_{\ka,q}$ reads
$$\UK(x_\indk)=\sum_{j\in\BKO}\E^{\I x_\indk jq}\hat\UK_{jq}.$$
Then, we use the explicit form \eqref{eq:DeltaK} of the flow of $\Delta^K$ in the Fourier basis. We get
\begin{eqnarray}
\E^{2\I\pi\frac{p}{q}\Delta^K} \ \UK(x_\indk)=\sum_{j\in\BKO}\E^{\I x_\indk jq}\E^{\I(jq)^2 2\pi\frac{p}{q}} \hat\UK_{jq}=\sum_{j\in\mathcal{B}^k}\E^{\I x_\indk jq+2\I\pi j^2pq} \hat \UK_{jq}=\UK(x_\indk).\nonumber
\end{eqnarray}

Finally, we get that the $\UK^{n}$ read
$$
\UK^{n}_k = \varphi_{2\pi \frac{p}{q}}(\UK^{n-1})_k = \varphi_{2n \pi \frac{p}{q}}(\UK^0_k) = \exp(-\sigma i n\tau |\UK^0_k|^2) \UK^0_k, 
$$
and obtain the first estimate of Lemma \ref{KeyLemmaDiscret}. 

To get the estimate of the $H^K(U^n)$, observe that for any $U$,
$$
H^K(U)= \pi T^K(\mathcal{F}_K U) +\sigma \delta x \sum_{k \in \BK} |U_k|^4 \geq c\Norm{U}{h^1}^2 - C\Norm{U}{h^1}\Norm{U}{\ell^2}^3,
$$
where the inequality comes from Lemma \ref{Lem:TK} together with Lemma \ref{Lem:GNdiscret}. Thus, from the previous estimate obtained on $\Norm{U^n}{h^1}$ and given that $\Norm{U^n}{\ell^2}$ is constant, we get
$$
H^K(U^n)\geq (c''n\tau -c''')^2.
$$
The proof for time steps of the form $\tau = 2\pi \frac{p}{q^2}$ is the same. 

\end{proof}

\end{document}